\documentclass[oneside,british,reqno]{amsart}
\usepackage{preamble}
\usepackage{subcaption}
\usepackage{import}
\usepackage{relsize}

\newcommand\X{{\bZ_0^d}}
\newcommand\Xz{{\bZ^d}}
\newcommand\Xk{{\bZ_k^d}}
\newcommand\act{\hat{+}}
\newcommand\kact{\hat{+}_k}

\newcommand\No{\bN\cup\set{0}}

\begin{document}
\title{Random walks on primitive lattice points}
\author{Oliver Sargent}
\begin{abstract}
    We define a random walk on the set of primitive points of $\bZ^d$. 
    We prove that for walks generated by measures satisfying mild conditions these walks are recurrent in a strong sense.
    That is, we show that the associated Markov chains are positive recurrent and there exists a unique stationary measure for the random walk.
\end{abstract}
\address{Faculty of Mathematics and Computer Science\\
    The Weizmann Institute of Science \\
    Rehovot \\
Israel }
\email{o.g.sargent@gmail.com}
\maketitle

\section{Introduction}
Random walks on lattices have been extensively studied by probabilists.
They provide fascinating examples and useful models for many natural processes.
Random walks are naturally part of the larger class of Markov chains.
One of the most basic properties of Markov chains is recurrence.
It is well known that in one or two dimensions symmetric walks are recurrent in the sense that they return to their starting points almost surely. 
In higher dimensions walks are transient which means that they are not recurrent in the above sense. For the details, see~\cite{MR0388547} or~\cite{MR2677157} for instance.

In this note we will define
random walks on primitive lattice points, or more generally, lattice points which lack specified factors 
common to all of their co-ordinates.
We will do this by considering the usual random walk on the lattice but at each stage we will divide by the `gcd of the vector' so that the walker always ends up at a primitive point.
It seems quite surprising that the author could not find a good reference for this simple and natural set up.
The properties of the primitive random walks we consider are very different from the usual random walks.
We will show that they exhibit strong recurrence properties. 
In the language of Markov chains we will show that they are positive recurrent.
In particular, we will show that there is an invariant stationary distribution. 
While at first this fact might seem quite surprising, in hindsight it can be seen to follow from a simple equidistribution property of random walks on the discrete torus.

\subsection{Main results}\label{ssec:mainresults}
To state our main results we must formally describe our setting.
Here and throughout the paper, by a lattice we mean a discrete subgroup of Euclidean space. 
Since any such object is isomorphic to $\bZ^d$ for some $d\in\bN\cup\set{0}$ we choose to restrict our attention to the lattice $\bZ^d$. 
Let $\set{e_1,\dots,e_d}$ be the standard basis of $\bR^d$ so we may write $\bZ^d=\Sp{e_1,\dots,e_d}_\bZ$. 
For $z\in\Xz$ and $1\le i\le d$ we define the 
co-ordinates $z_i\defi z\bcdot e_i$.
We recall that $z\in\bZ^d$ is said to be \emph{primitive} if $\gcd(z_1,\dots,z_d)=1$. Let $\bZ_0^d$ denote the set of primitive points of $\bZ^d$.
We will use the convention that $\gcd(0,\dots,0)=1$ and so $(0,\dots,0)\in\X$.

Consider the map $\bZ^d\times \bZ_0^d \to \bZ_0^d$ given by 
\begin{equation}\label{eq:defofact}(a,z)\mapsto a\hat{+}z\defi\frac{a+z}{\gcd(a+z)}.\end{equation} 
We remark that this map is not an action of $\Xz$ on $\X$ since there exists $a_1,a_2\in\bZ^d$ and 
$z\in\X$ such that $(a_1+a_2)\act z \neq a_1\act a_2\act z$.

To describe our main results as quickly as possible we use the language of Markov chains. 
We refer the reader to \cite[Chapter VIII]{MR737192} for the basic facts we will use.
For a probability measure $\mu$ on $\Xz$ 
we consider the random walk on $\Xz$ where the walker chooses each step with probability given by $\mu$ and then moves to the position given by the map in~\eqref{eq:defofact}.
Suppose that the walker starts at $z\in\X$, then we let $\set{\cX_{i,z}}_{i=0}^\infty$ be the sequence of random variables which give the position of the walker after $i=0,1,2,\dots$ steps.
We note that the sequence of random variables $\set{\cX_{i,z}}_{i=0}^\infty$ forms a Markov chain with transition probabilities given by 
%
%
\begin{equation*}\mb P_{\mu}[\cX_{i+1,z}=x\act y|\cX_{i,z}=y]\defi\mu(x)\qfa\ y\in\X\ \textrm{and}\ i\in\No\ :\ \mb P_\mu[\cX_{i,z}=y]>0.\end{equation*}
We use the notation $\on M_z^0(\mu)$ to denote this Markov chain.

A Markov chain is said to \emph{irreducible} (or sometimes indecomposable) if there is a positive probability to reach a specified state from any given starting state. 
An irreducible Markov chain is said to be \emph{positive recurrent} if for all states, the expected return time to that state is positive.

For a probability measure $\mu$ on $\Xz$ we will use the following two standing assumptions which we record here for convenience:
\begin{enumerate}[\enskip (A)]
    \item\label{assump:0} Finite first moment, in the sense that $\sum_{z\in\bZ^d}\norm{z}\mu(z)<\infty.$
\item\label{assump:1} Support which generates $\bZ^d$, in the sense that for every $z\in\Xz$ there exists $n\in\bN$ such that $\mu^{\ast n}(z)>0$.
\end{enumerate}
Our first main result is the following:
\begin{thm}\label{thm:MC1}
    Let $\mu$ be a probability measure on $\Xz$ satisfying~\eqref{assump:0} and~\eqref{assump:1}, then for all $z\in\X$ the Markov chain $\on M_z^0(\mu)$ is irreducible and positive recurrent.
\end{thm}
The assumption that $\mu$ satisfies~\eqref{assump:0} 
is necessary to ensure that the random walk does not spread out too fast.
Assumption~\eqref{assump:1} is a nondegeneracy assumption which ensures that
the Markov chain is irreducible. 

In order to prove Theorem~\ref{thm:MC1} 
we will consider a family of inherently less recurrent Markov chains and show that they are all 
positive recurrent. 
This family of Markov chains can be thought of as random walks on the set of lattice points \emph{coprime} to $k$, for some $k\in\bN$.
We say that a point $z\in\Xz$ is coprime to $k\in\bN$ if $k\nmid\gcd(z)$.
We denote the set of points in $\Xz$ which are coprime to $k$ by $\Xk$. 
Consider the map $\Xz\times \Xk\to\Xk$ given by 
\begin{equation}\label{eq:defofkact}(z,x)\mapsto z\hat{+}_kx\defi\frac{z+x}{k^p},\quad\textrm{where}\ p\defi\max\set{n\in\bN:k^n\mid (z+x)}.\end{equation}
We proceed in a similar manner to how we constructed the Markov chain $\on M_z^0$.
For a probability measure $\mu$ on $\Xz$ we let $\on M_z^k(\mu)$ denote the Markov chain corresponding to the walk obtained by starting at $z\in\Xk$ and 
iterating the map 
in~\eqref{eq:defofkact} with steps chosen according to the measure $\mu$.
That is, $\on M_z^k(\mu)\defi\set{\cX_{i,z}}_{i=0}^\infty$ a sequence of random variables, corresponding to the random walk starting at $z$ with transition probabilities given by 
\begin{equation*}\mb P_{\mu}[\cX_{i+1,z}=x\kact y|\cX_{i,z}=y]\defi\mu(x)\qfa\ y\in\Xk\ \textrm{and}\ i\in\No\ :\ \mb P_\mu[\cX_{i,z}=y]>0.\end{equation*}
Since there is less likely hood of division occurring in the latter set up, 
one should expect that the Markov chains $\on M_z^k$ are less recurrent than $\on M_z^0$ for any $k\in\bN$ with $k\geq 2$. 
We will see that this is indeed the case and we 
we will prove the following result.
\begin{thm}\label{thm:MC2}
    Let $\mu$ be a probability measure on $\Xz$ satisfying~\eqref{assump:0} and~\eqref{assump:1}, then for all $z\in\Xk$ and $k\in\bN$ with $k\geq 2$ the Markov chain $\on M_z^k(\mu)$ is irreducible and positive recurrent.
\end{thm}
The irreducibility implies that the state space of the Markov chain $\on M_z^k$ is $\Xk$ for all $z\in\Xk$. 
Hence, we see from the definition of positive recurrence that if $\on M_z^k$ is positive recurrent for some $z\in\Xk$ then it will be positive recurrent for all $z\in\Xk$.
\subsection{Invariant measures}\label{ssec:invmus}
The strategy used to prove Theorems~\ref{thm:MC1} and~\ref{thm:MC2} is first to show that for measures satisfying~\eqref{assump:1} the Markov chains $\on M_k(\mu)$ are irreducible for all $k\in\bN\cup\set{0}$.
For a time homogeneous Markov chain $\on M=\set{\cX_i}_{i=0}^\infty$ consisting of $X$ valued random variables with transition probabilities given by 
\begin{equation*}\mu_x(E)\defi\mb P[\cX_{i+1}\in E|\cX_i=x]\qfa\ x\in X,\ E\subseteq X\ \textrm{and}\ i\in \No\ :\ \mb P_\mu[\cX_i=x]>0\end{equation*} 
a measure $\nu\in\cP(X)$ is said to be a \emph{stationary measure} (or sometimes invariant distribution) for $\on M$ if 
$\int_X\mu_x\dv{\nu}x=\nu.$

It is then a classical fact~\cite[Theorem 2, p. 543]{MR737192} that irreducible Markov chains with countable state spaces are positive recurrent if and only if there exists a stationary measure for the Markov chain. 
Moreover, the irreducibility implies that if a stationary measure exists then it is unique.

There is an obvious candidate for such a stationary measure and this is the object which we will study in the rest of this note.

First let us
introduce the 
notation $\cP(X)$ for the set of probability measures on $X$, for some topological space $X$. 
Let $B\defi(\Xz)^\bN$ be the space of infinite sequences of elements in $\Xz$ and $\be\defi\mu^{\otimes \bN}\in\cP(B)$ the Bernoulli measure.

Given $k\in\bN\cup\set{0}$, $b\in B$ and $z\in\X$ we will consider the random walks corresponding to $b$ starting at $z$ given by the sequences
\begin{equation}\label{eq:rw} \set{b_n\act \dots \kact b_1 \kact z }_{n\in\bN},\end{equation}
where $\hat{+}_0\defi\act$.
In order to study stationary measures for these random walks and the associated Markov chains we consider the measures supported on the sets of points which are end points of walks of length $n$.
For $n\in\bN$ and $z\in\X$ we define $\om^0_{n,z}\in\cP(\X)$ to be the measure  
\begin{equation}\label{eq:defofnthtrans}\om^k_{n,z}(E)=\int_B \bone_E(b_n\kact\dots\act b_1\act z)\dv{\be}b\qfa\ E\subseteq \Xk.\end{equation}
Now the natural candidates for invariant measures for the Markov chains $\on M_k(\mu)$ are given by weak-* limits of the following sequence 
\begin{equation}\label{eq:defofinvmu}
\set*{\frac{1}{n}\sum_{i=1}^n\om^k_{i,z}}_{n\in\bN}.\end{equation}
Using the definition of an invariant measure one can see that such limit points will be stationary.
Therefore, the main difficulty is to show that the any weak-* limit of the sequence in~\eqref{eq:defofinvmu} is a probability measure. 
One often thinks of this as being a `nonescape of mass' property of the sequence of measures~\eqref{eq:defofinvmu}. 
%
%
We will prove the following theorem.
\begin{thm}\label{thm:main1}
Suppose that $\mu\in\cP(\bZ^d)$ satisfies~\eqref{assump:0} 
and~\eqref{assump:1}.
Then, for all $k\in\bN\cup\set{0}$ and $z\in\X$ any weak-* limit of the sequence in~\eqref{eq:defofinvmu} is a stationary probability measure for the Markov chain $\on M_z^k(\mu)$.
Moreover, there is unique limit point which we denote $\om^k_\infty$.
\end{thm}
It follows from the discussion at the start of this subsection that once we have shown irreducibility (c.f. \S\ref{ssec:irred}), Theorems~\ref{thm:MC1} and~\ref{thm:MC2} follow directly from Theorem~\ref{thm:main1}.
We also remark that the stationary measure satisfies $\om^k_\infty(z)=1/\ta_k(z)$ for all $z\in\Xk$ where $\ta_k(z)$ is the expected return time of the random walk (corresponding to $k$) to $z$.
Hence, it would be interesting to obtain estimates for $\ta_k(z)$.
\subsection{Figures}        
Using a computer algebra package it is possible to simulate fairly long walks on the spaces $\Xk$ and $\X$.
The following (`randomly picked') measures were used to generate random walks.
\begin{align*}
\eta_1 \defi &\frac{1}{200}(13\del_{e_1}+3\del_{e_2}+35\del_{e_3}+36\del_{e_4}+36\del_{-e_1}+30\del_{-e_2}+42\del_{-e_3}+5\del_{-e_4})\\
\eta_2 \defi &\frac{1}{305}(13\del_{e_1} +3\del_{e_2}+ 35\del_{e_3}+36\del_{e_4} +5\del_{e_5} +42\del_{e_6}\\
& +16\del_{-e_1} +36\del_{-e_2}+4\del_{-e_3}+49\del_{-e_4}+36\del_{-e_5}+30\del_{-e_6})\\
\eta_3 \defi &\frac{1}{51}(11\del_{e_1}+12\del_{e_2}+8\del_{e_3}+9\del_{-e_1}+2\del_{-e_2}+9\del_{-e_3})\\
\end{align*}
In order to visualise the random walks, in particular the recurrence properties of the random walk, the norms of all points in the walks were calculated.
This data was then used to plot histograms as displayed in Figure~\ref{fig:3}.

\begin{figure}[h]
\begin{subfigure}{0.49\textwidth}
\centering
\includegraphics[height=12em]{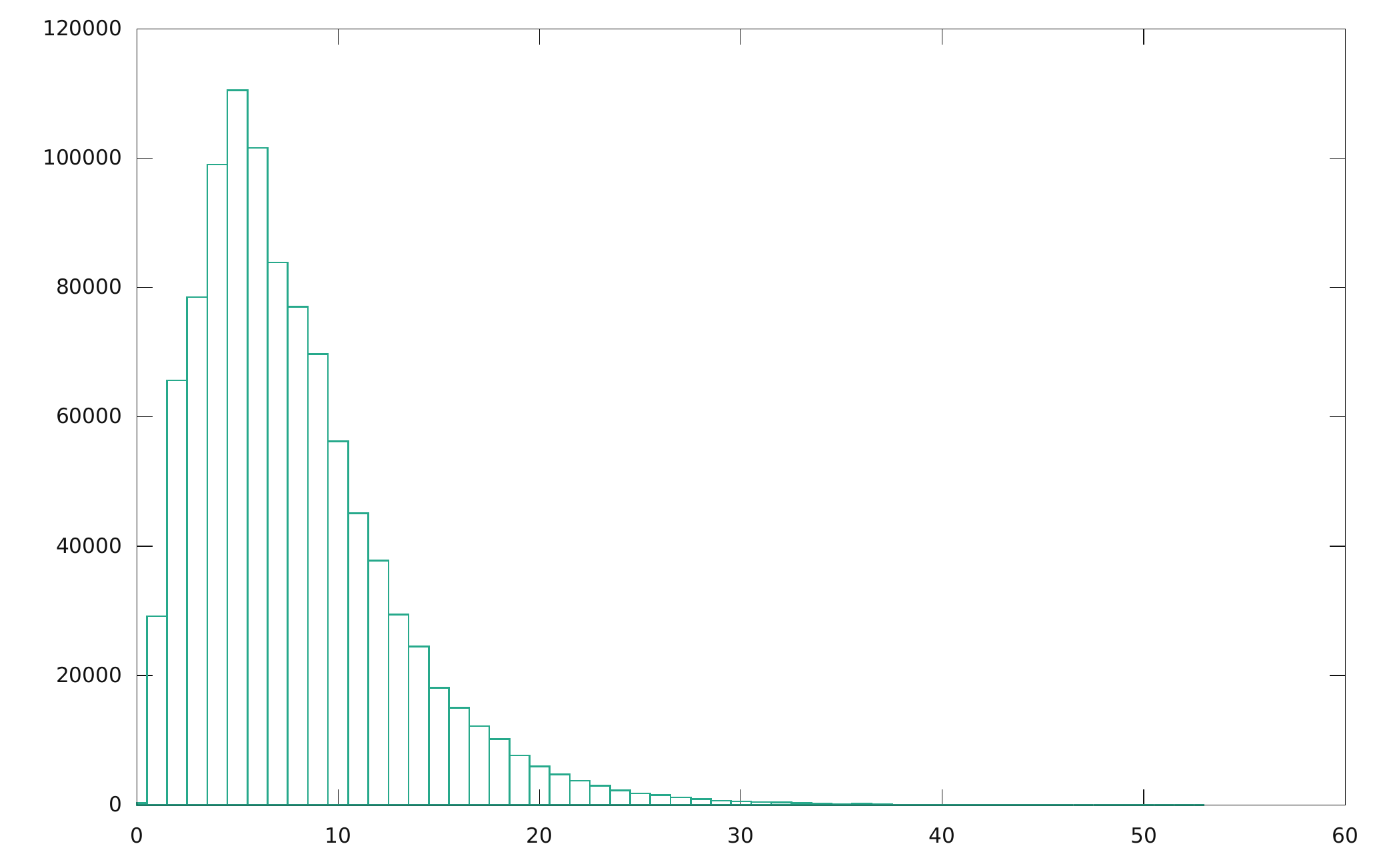}
\caption{Walk on $\bZ^4_0$ generated by $\eta_1.$} 
\label{fig:subim1}
\end{subfigure}
\begin{subfigure}{0.49\textwidth}
\centering
\includegraphics[height=12em]{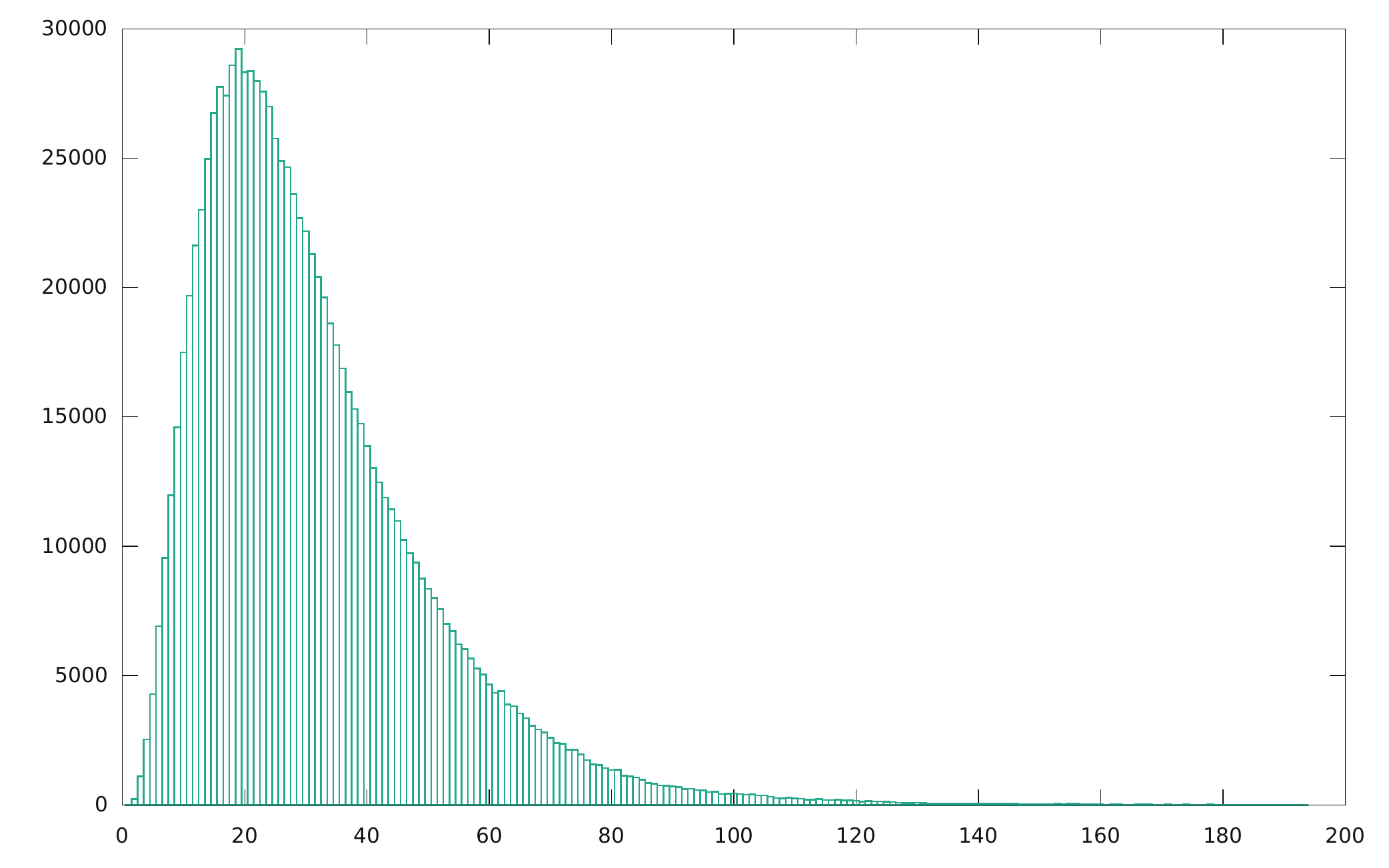}
\caption{Walk on $\bZ^6_0$ generated by $\eta_2$.}
\label{fig:subim2}
\end{subfigure}
\begin{subfigure}{0.49\textwidth}
\centering
\includegraphics[height=12em]{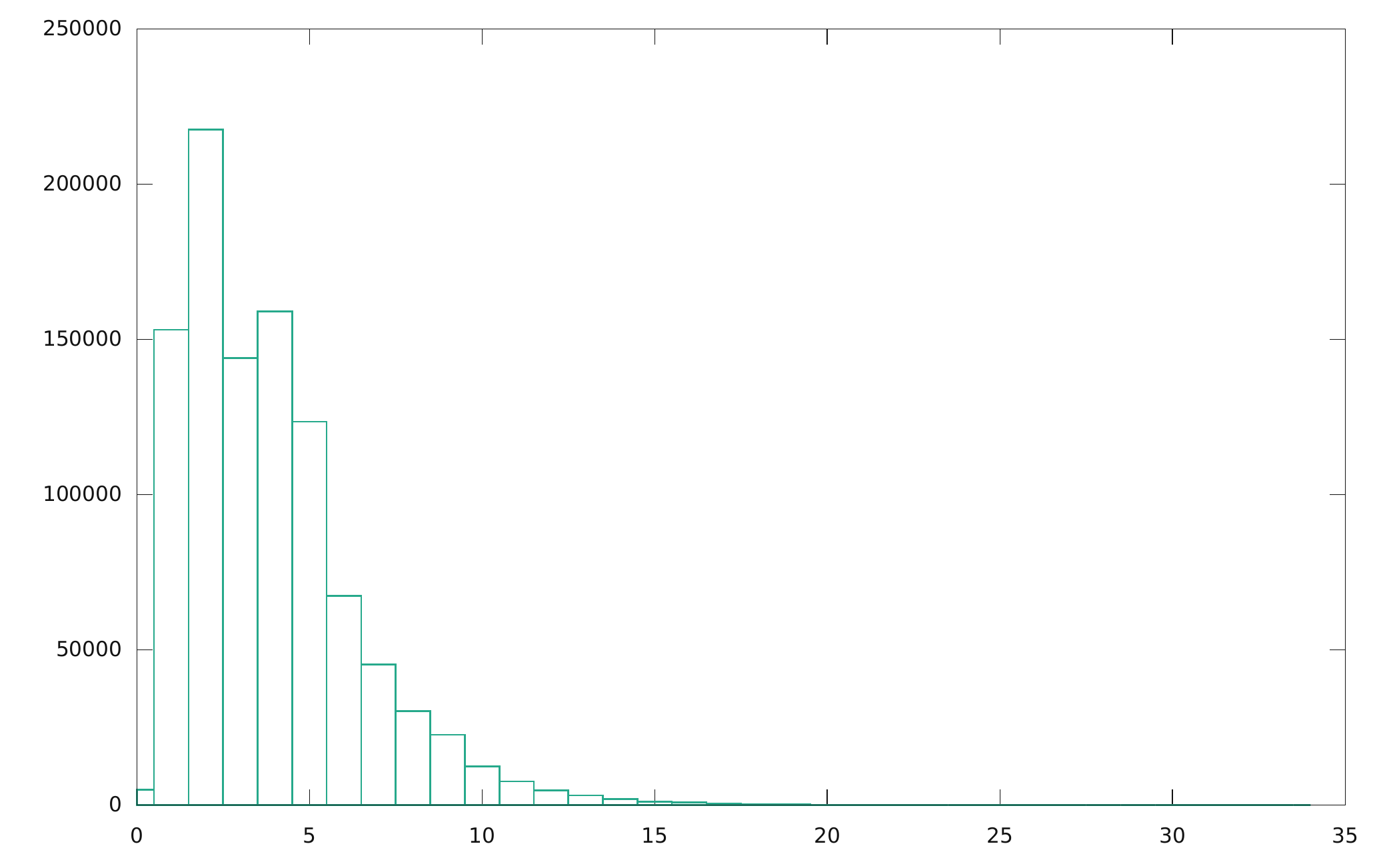}
\caption{Walk on $\bZ_2^3$ generated by $\eta_3$.}
\label{fig:subim3}
\end{subfigure}
\begin{subfigure}{0.49\textwidth}
\centering
\includegraphics[height=12em]{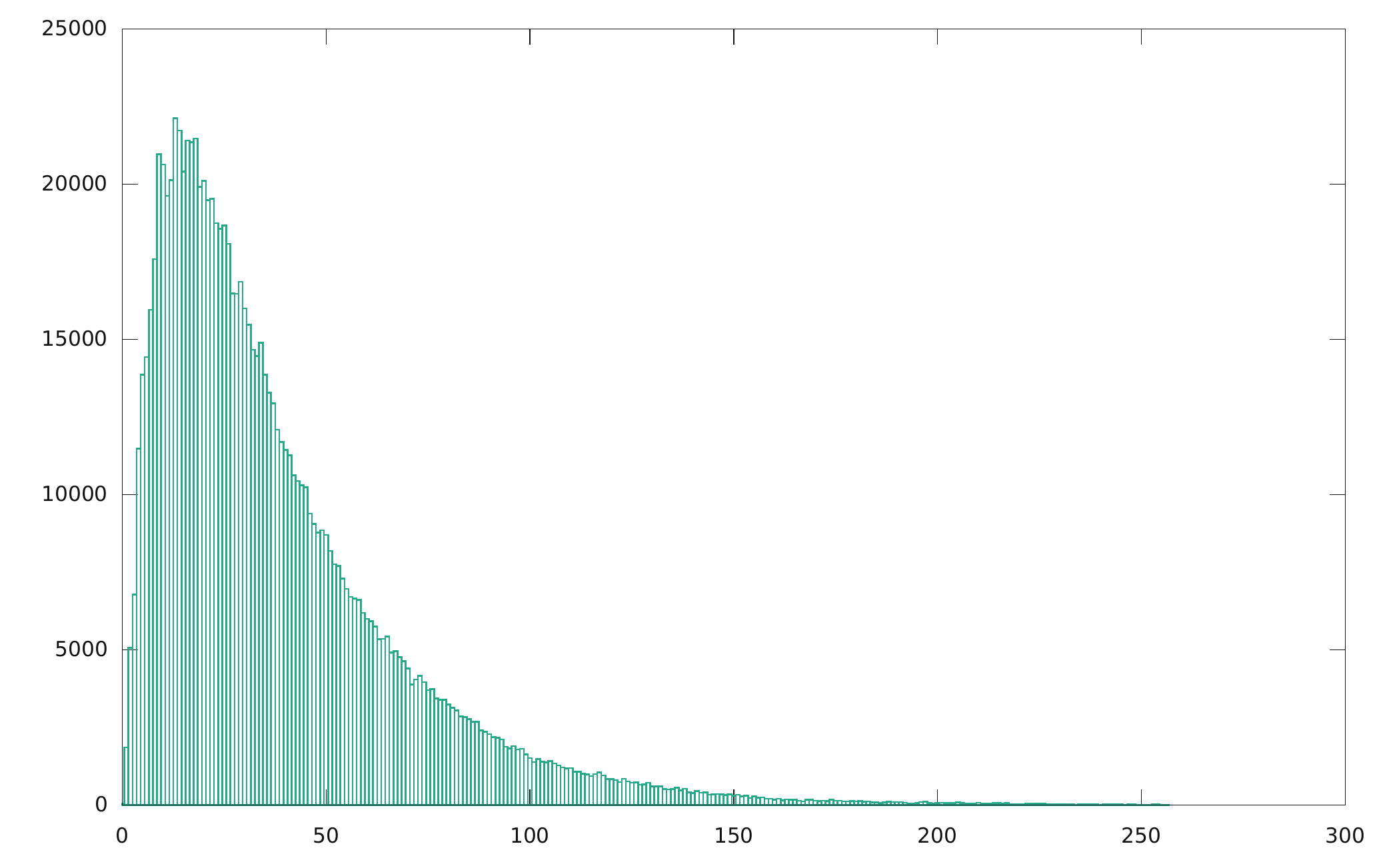}
\caption{Walk on $\bZ_5^3$ generated by $\eta_3$.}
\label{fig:subim4}
\end{subfigure}
\caption{Histograms showing frequencies of norms of vectors in the first 1,000,000 points of various walks.
The frequencies are indicated on the vertical axis and the norms of vectors in the walk on the horizontal.}
\label{fig:3}
\end{figure}

\section{Proof of main results}
In this section we will complete the proofs of Theorems~\ref{thm:MC1} and~\ref{thm:MC2}.
As previously remarked, this is done in two steps; first we show that the Markov chains are irreducible and second we will prove Theorem~\ref{thm:main1}.

In~\S\ref{ssec:irred} we will 
see how assumption~\eqref{assump:1} is used to show that the corresponding Markov chain is irreducible.
In~\S\ref{ssec:rec} we will introduce a strong recurrence property for random walks and show that if a walk has this properties then the 
limit in~\eqref{eq:defofinvmu} converges.
In~\S\ref{ssec:disctor} we study random walks on the discrete torus. 
Using the fact that these walks tend to be equidistributed we derive a lower bound for the expected number of times division by a prime should
occur in a walk of length $n$.
In~\S\ref{ssec:contnorms} we use the bounds from the previous section to show that on average the norm is contracted by the random walks.
This fact will then be used to show that the random walks satisfy the recurrence property introduced in~\S\ref{ssec:rec}.
%
\subsection{Irreducibility}\label{ssec:irred}
We recall that a Markov chain $\on M=\set{\cX_i}_{i=0}^\infty$ consisting of $X$ valued random variables for some countable state space $X$ is said to be \emph{irreducible} if for all $x,y\in X$ with $\mb P\br{\cX_0=x}>0$, there exists $n\in\bN$ such that $\mb P[\cX_n=y|\cX_0=x]>0$.

For $k,n\in\bN\cup\set{0}$, $z\in \Xk$ and $b\in B$ let $\Sig_n^k(b,z)$ be the position of the random walk in $\Xk$ 
corresponding to the sequence $b$ starting at $z$ after $n$ steps.
In other words
\begin{equation}\label{eq:defofsig}\Sig_n^k(b,z)\defi b_n\kact\dots\kact b_1\kact z,\end{equation}
where as before $\hat{+}_0\defi\act$.
We view $\Sig_n^k(-,z):B\to\Xk$ as random variables so that 
$\on M_z^k(\mu)=\set{\Sig_i^k(-,z)}_{i=0}^\infty$ and
\begin{equation}\label{eq:nstepprob}\mb P_{\mu}[\Sig_n^k(b,z)=x]=\mu^{\ot n}\mset{a\in (\Xz)^n:a_n\kact\dots\kact a_1\kact z=x},\end{equation}
where 
$\on M_z^k(\mu)$ is the Markov chain defined in~\S\ref{ssec:mainresults}.

We start with the following general lemma.
\begin{lem}\label{lem:abscont}
    Let $k\in\bN\cup\set{0}$, $z\in\Xk$ and $\on M_z^k$ be the Markov chains defined in~\S\ref{ssec:mainresults}.
If $\mu_1,\mu_2\in\cP(\bZ^d)$ such that $\mu_1$ is absolutely continuous with respect to $\mu_2$ and $\on M_z^k(\mu_1)$ is irreducible. Then $\on M_z^k(\mu_2)$ is irreducible.
\end{lem}
\begin{proof}
    First note that for $\mu_1$ and $\mu_2$ satisfying the hypothesis of the lemma,
    $\mu_1^{\otimes n}$ is absolutely continuous with respect to $\mu_2^{\otimes n}$ for all $n\in\bN$.
    For a contradiction, suppose that $\on M_z^k(\mu_2)=\set{\cX_{i,z}}_{i=0}^\infty$ is not irreducible and thus, there exists $x,y\in\Xk$ such that $\mb P_{\mu_2}[\cX_{0,z}=x]>0$ and 
    \begin{equation*}\mb P_{\mu_2}[\cX_{i,z}=y|\cX_{0,z}=x]=0\qfa\ i\in\bN.\end{equation*}
    This implies that $x=z$ and hence $\mb P_{\mu_2}[\cX_{i,z}=y|\cX_{0,z}=x]= \mb P_{\mu_2}[\Sig_i^k(b,z)=y]$.
    Thus, using~\eqref{eq:nstepprob} we see that 
    \begin{equation*}
        \mu_2^{\otimes i}\mset{a\in (\Xz)^i:y=a_i\kact\dots\kact a_1\kact x}=0\qfa\ i\in\bN
    \end{equation*}
    and hence using the absolute continuity
    \begin{equation*}
        \mu_1^{\otimes i}\mset{a\in (\Xz)^i:y=a_i\kact\dots\kact a_1\kact x}=0\qfa\ i\in\bN.
    \end{equation*}
    Reversing the argument we see that this implies that $\mb P_{\mu_1}[\cX_{i,z}=y|\cX_{0,z}=x]=0$ for all $i\in\bN$. 
    Since $\on M_z^k(\mu_1)$ is assumed to be irreducible, this is a contradiction as required.
\end{proof}
Now we are ready to prove our main result of this subsection.
\begin{prop}\label{prop:irred}
    For all $k\in\No$ and $z\in\Xk$ and measures 
    $\mu\in\cP(\bZ^d)$ satisfying~\eqref{assump:1} the Markov chain $\on M_z^k(\mu)$ is irreducible.
\end{prop}
\begin{proof}
    One easily sees from the definition of irreducibility, that if the Markov chain $\on M_z^k(\mu)$ is irreducible for some $z\in\Xk$ then it is 
    irreducible for all $z\in\Xk$. 
    So we will set $z=0$. 
    Moreover, we will only treat the case when $k=0$, but it can be easily checked that the proof goes through without change for $k\in\bN$.
    Set $\on M\defi\on M_0^0$.
    
    Since we suppose that $\mu\in\cP(\bZ^d)$ satisfies assumption~\eqref{assump:1}, 
    there exists $n\in\bN$ such that $\nu\defi\frac{1}{2d}\sum_{i=1}^d\delta_{\pm e_i}$ is absolutely continuous with respect to
$\lam_n\defi\frac{1}{n}\sum_{i=1}^n\mu^{*i}$.
Therefore, if $\on M(\nu)$ is irreducible, by Lemma~\ref{lem:abscont} we see that $\on M(\lam_n)$ is irreducible.
It is easy to see that irreducibility of $\on M(\lam_n)$ implies the irreducibility of $\on M(\mu)$.

Hence, in order to complete the proof of the proposition, we must show that $\on M(\nu)=\set{\cX_i}_{i=0}^\infty$ is irreducible.
Or, in other words that for all $x,y\in\X$ such that $\mb P_\nu[\cX_0=x]>0$ there exists $n\in\bN$ such that $\mb P_{\nu}\br{\cX_n=y|\cX_0=x}>0$. 
Since $\mb P_\nu[\cX_0=x]>0$ if and only if $x=0$ we may assume that $x=0$.

We say that $y\in\X$ is \emph{connected to 0} if
there exists $n\in\bN$ and a sequence of points $\set{p_i}_{1\leq i\leq n}\subset\X$ such that $p_1=0$, $p_n=y$ and 
$p_i-p_{i+1}\in\set{\pm e_1,\dots,\pm e_d}$ for all $1\leq i\leq n-1$.

So we must show that every $x\in\X$ is connected to 0.
Suppose that $d\geq 3$. 
Since $\gcd(1,z)=1$ for all $z\in\bZ^{d-1}$, if $x=(1,x_2,\dots,x_d)$ then $x$ is connected to 0. 
Moreover, if $x\in\X$ is such that $\gcd(x_2,\dots,x_d)=1$ then as in the previous case $(1,x_2,\dots,x_d)$ is connected to 0 and 
we may connect $(1,x_2,\dots,x_d)$ with $x$ by adding or subtracting $e_1$ an appropriate number of times.
In the case that $x\in\X$ is such that $\gcd(x_2,\dots,x_d)>1$, since $d\geq 3$ we have $\gcd(x_2+1,x_3,\dots,x_d)=1$. 
Hence, by our previous argument $x+e_2$ is connected to 0, but this clearly implies that $x$ is connected to 0 and hence in the case that $d\geq 3$ we have shown that
every $x\in\X$ is connected to 0 as required.

Suppose that $d=2$. Let $x\in\bZ_0^2$ be arbitrary. 
There exists a prime $p$ and $n\in\bN$ such that $p=x_2 n+1$. 
It is clear that $(1,p)$ is connected to 0 and we may connect $(1,p)$ with $(nx_1,p)$ by adding or subtracting $e_1$ an appropriate number of times.
Since $(nx_1,p)\act (-e_2)=(x_1,x_2)$ this shows that $\mb P[\cX_n=x|\cX_0=0]>0$ for some $n\in\bN$.
To see that $\mb P[\cX_n=0|\cX_0=x]>0$ we note that for all $x\in\X$, there exists $n\in\bN$ such that 
the element $(1,x_2)$ occurs in the sequence $\set{x\act ie_1}_{\abs{i}\leq n}$. 
Without loss of generality, suppose that $x_2\leq 0$ so that $e_2\act(-e_1)\act(1,x_2)=(0,0)$ and hence we see that $\mb P[\cX_n=0|\cX_0=x]>0$ as required.
\end{proof}

\subsection{Recurrence}\label{ssec:rec}
Deciding if a random walk is recurrent or transient is one of the first basic steps towards understanding the long term behaviour of the random walk.
There are several related notations of recurrence of a random walk.
Following~\cite{MR2874934}, in this note we are interested in the following strong notion:
\begin{defn}\label{def:rec}    
For $k\in\bN\cup\set 0$ and $\mu\in\cP(\Xz)$, we say that the random walk generated by $\mu$ on $\Xk$ is \emph{recurrent} if for all $z\in \Xk$ and $\eps>0$ 
there exists a finite set $K_{\eps}\subset \Xk$ and $n_z\in\bN$ such that for all $n\geq n_z$ we have 
$$\om^k_{n,z}(K_{\eps})>1-\eps.$$ 
\end{defn}
The proof of Theorem~\ref{thm:main1} can be reduced to proving the following proposition.
\begin{prop}\label{prop:redux}
For all $k\in\bN$ and $\mu\in\cP(\Xz)$ satisfying~\eqref{assump:0} and~\eqref{assump:1} we have that the random walk induced by $\mu$ on $\Xk$ is recurrent.
\end{prop}
The proof of Proposition~\ref{prop:redux} will be our goal for the remainder of this note.
First we show that Theorem~\ref{thm:main1} immediately follows from it.
\begin{proof}[Proof of Theorem~\ref{thm:main1} assuming Proposition~\ref{prop:redux}]\label{rem:redux}
First we note that Proposition~\ref{prop:redux} implies that the walk generated by $\mu$ on $\X$ is recurrent. 
To see this, note that for all $n\in\bN$ and $z\in\X$ one has $\om^0_{n,z}(K)\ge\om^k_{n,z}(K)$ for all cones $K\subset \X\subset\Xk$.
Moreover, any finite subset of $\X$ can be contained in a cone and therefore one may assume that the subset $K_\eps$ in Definition~\ref{def:rec} is a cone.

It follows that for all $k\in\bN\cup\set 0$ and $z\in\Xk$ any weak-*
accumulation point of the sequence $\set{\frac{1}{n}\sum_{i=1}^n\om^k_{i,z}}_{n\in\bN}$ will be a probability measure. 
It is also clear that any such limit point will be a stationary measure for the Markov chain $\on M_z^k(\mu)$. 
Since $\Xk$ is countable and by Lemma~\ref{prop:irred} the Markov chains corresponding to the random walks are irreducible, 
it follows from~\cite[Theorem 2, p. 543]{MR737192} that the limit point 
will actually be unique and this is the claim of Theorem~\ref{thm:main1}.
\end{proof}
We remark that the fact that $\mu$ satisfies both of~\eqref{assump:0} and~\eqref{assump:1} is vital for the statement of Proposition~\ref{prop:redux} to hold.
As we saw in~\S\ref{ssec:irred},~\eqref{assump:1} is needed to guarantee irreducibility, but this does not mean it is not needed for Proposition~\ref{prop:redux}.
Indeed, if it was simply removed, it is possible that for certain starting locations $z\in\Xk$, the random walk corresponding to $\mu$ never visits points with a gcd divisible by $k$ and
hence the walk on the set of points coprime to $k$ starting at $z$ would behave more like a traditional random walk on a lattice.


\subsection{Equidistribution on the discrete torus}\label{ssec:disctor}
In this section we will study the ordinary random walk on $\Xz$.
In order to prove recurrence we must exploit the fact that for a long walk, a positive proportion of sites that it visits will have
a common divisor which is divisible by $k$.
In order to make this precise we study the corresponding random walk on the discrete torus. The above statement then just becomes an equidistribution property for such walks. 
For $n\in\bN$, $z\in\Xz$ and $b\in B$ we denote the position of the random walk in $\Xz$ corresponding to $b$ after $n$ steps by 
$$\Sig_n(b,z)\defi b_n+\dots+ b_1+ z.$$

We view the functions $\Sig_n(-,z):B\to\Xz$ as random variables.

We will be interested in how many times during the first $n$ steps of the walk we saw a $\on{gcd}>1$.
The aim is to use the fact that on the discrete torus after walking for a very long time the past and future tend to become independent of one another.
This will enable us to use results from probability theory concerning deviations from the expected value for sums of independent random variables.
With this in mind we will define some more random variables.
For $k\in\bN$, $n\in\bN$ and $z\in\Xz$ let
$$\cY^k_n(b,z)\defi\abs{\set{1\le i\le n:\Sig_i(b,z)\equiv 0\mod k}},$$
where for all $z\in\Xz$ we say that $z\equiv 0\mod k$ if $k\mid\gcd(z)$ or if $z=0$.
We view $\cY^k_n(-,z)$ as a random variable which records the number of times in the first $n$ steps of the random walk a common factor of $k$ appears. 
Let $\mb 1$ be the constant function with value 1 on $B\times\Xz$ and 
$$\cM^k_n(b,z)\defi\min\set{\cY^k_n(b,z),\mb 1}=
\begin{cases*}
1 & if $\cY_n^k(b,z)\geq 1$ \\
0 & otherwise.
\end{cases*}$$
Moreover, let $\DT k d$ denote the discrete $d$-dimensional torus of width $k$. 
We can identify 
\begin{equation}\label{eq:ident torus in zd}
\DT k d\cong\set{0,\dots,k-1}^d\subset\Xz 
\end{equation}
in the usual manner.
Let $$\cU_n^k(b)\defi\min\set{\cM_n^k(b,x):x\in \DT k d}.$$
As is customary, for measurable functions $f:B\to\bR$ and $E\subset\bR$ we use the notation 
$$\mb E[f]\defi\int_B f\dv\be{}\qand\mb P[f\in E]\defi\int_B \bone_{\set{b\in B:f(b)\in E}}\dv\be{}.$$
\begin{lem}\label{lem:uniform lower bound}
For all $k\in\bN$ there exists $n_0>0$ such that for all $n\geq n_0$ one has 
$$\mb{E}[\cU^k_n]>0.$$
\end{lem}

\begin{proof}
By definition we have that 
\begin{align}
    \nonumber\mb{E}[\cU_n^k]=\mb{E}[\min\set{\cM_n^k(b,z):z\in \DT k d}]&=\int_B\min_{z\in\DT k d}\bone_{\set{b\in B:\cY_n^k(b,z)\ge 1}}\dv\be{} \\
\label{eq:min is cap}                                                         &=\be(\set{b\in B:\cY_n^k(b,z)\ge 1\ \textrm{for all}\ z\in\DT k d}).
\end{align}
Since the measure $\mu$ generates $\bZ^d$ by assumption~\eqref{assump:1} there exists 
$n_0\in\bN$ and a finite sequence $$a\defi (a_1,\dots,a_{n_0})\in (\supp \mu)^{n_0}\subset(\Xz)^{n_0}$$ such that 
$\set{\sum_{i=1}^{n}a_i:1\leq n\leq n_0}$ contains a copy of $\DT k d$ after using the identification~\eqref{eq:ident torus in zd}.
Moreover, 
\begin{equation}\label{eq:cyclinder is pos}
\be(C(a))>0,
\end{equation}
where
$$C(a)\defi\set{b\in B:b_i=a_i\ \fa\ 1\le i\le n_0},$$
is the cylinder set of $a$.
It follows that for all $b\in C(a)$,
$z\in \DT k d$ and $n\ge n_0$ one has 
$\cY_n^k(b,z)\ge 1$.
Hence
$$C(a)\subset\set{b\in B:\cY_n^k(b,z)\ge 1\ \fa\ z\in\DT k d}\qfa\ n\ge n_0.$$
Hence the claim of the lemma 
follows from~\eqref{eq:min is cap} and~\eqref{eq:cyclinder is pos}.
\end{proof}
In the following lemma we will use the notation $S$ for the shift map $S:B\to B$ given by $Sb=S(b_1,\dots)=(b_2,\dots)$.
\begin{lem}\label{lem:postive proportion of primes}
There exists $n_0>0$ such that for all $0<\eps<1$, $z\in \Xz$ and $n\ge n_0$ one has 
$\alpha\defi\mb{E}[\cU_{n_0}^k]>0$ and
$$\mb{P}\Bigl[\cY^k_n(-,z)\leq  \frac{(1-\eps)\alpha}{2n_0} n\Bigr]\le C_{\eps,\al}\exp\Bigl(-\frac{\al\eps^2}{2n_0}n\Bigr),$$
where $C_{\eps,\al}\defi\exp(\al\eps^2/2)$.
\end{lem}
\begin{proof}
Let $n_0\in\bN$ be large enough so that the conclusion of Lemma~\ref{lem:uniform lower bound} holds for all $n\geq n_0$. 
Let $m\in\bN$ and suppose that 
$mn_0\le n\le (m+1)n_0$. 
It follows from the definitions that 
\begin{equation}\label{eq:uniformisation}
\cY^k_n(b,z)\geq \sum_{i=1}^m \cM^k_{n_0}(S^{in_0}b,\Sig_{in_0}(b,z))
\geq\sum_{i=1}^m\cU^k_{n_0}(S^{in_0}b)
\end{equation}
for all $z\in \Xz$ and $b\in B$.
The set of random variables $\set{\cU_{n_0}^k\circ S^{in_0}}_{i\in\bN}$ consists of pairwise independent elements.
By Lemma~\ref{lem:uniform lower bound} and the fact that $\be$ is $S$-invariant we have that 
$\mb{E}\br{\cU_{n_0}^k\circ S^{in_0}}=\mb{E}\br{\cU_{n_0}^k}=\al>0$ for all $i\in\bN$.
It follows from the Chernoff bound (See~\cite[Theorem 4.5]{MR2144605}.) that 
$$\mb{P}\Big[\sum_{i=1}^m\cU^k_{n_0}\circ S^{in_0}\leq(1-\eps)\al m\Big]\le\exp\pa[\Big]{-\frac{\al\eps^2}{2}m}.$$
By \eqref{eq:uniformisation} we have that
$$\mb{P}\Big[\cY^k_n(-,z)\leq   c\frac{n}{n_0}\Big]\le\mb{P}\Big[\sum_{i=1}^m\cU^k_{n_0}\circ S^{in_0}\leq c\frac{n}{n_0}\Big]\qfa\ x\in X,$$
for all $c>0$.
Take $c\defi\al(1-\eps)/2$. Then $cn/n_0\leq\al(1-\eps)m$ where we use the facts that $1/2\le m/(m+1)$ for all $m\in\bN$ and $n/n_0\le m+1$.
Then the claim of the lemma follows from the previous equations where we use again the fact that $m\ge(n/n_0-1)$.
\end{proof}

\subsection{The norm is contracted}\label{ssec:contnorms}
In this subsection we use the results of the previous subsection to show that the norm is contracted by averaging with respect to the measures $\om^k_{n,z}$ for large enough $n$.
This will enable us to show that the random walks are recurrent as in Definition~\ref{def:rec}. 
Note that we follow closely~\cite[\S2]{MR2874934}. The reason we do not cite their results directly is that we are not dealing with a group action. 
At the end of the subsection we will complete the proof of Proposition~\ref{prop:redux}.

For $k\in\bN$, $n\in\bN$, $z\in \Xk$ and $b\in B$ 
recall the definition of $\Sig_n^k(b,x)$ from~\eqref{eq:defofsig}.
%
Note that the functions $\Sig_n^k(-,z):B\to\Xk$ only depend on the first $n$ co-ordinates of $b\in B$.
In other words they are measurable with respect to the $\sig$-algebra generated by the cylinder sets $\set{C(a):a\in (\Xz)^n,\ n\in\bN}$.
Therefore, we can consider $\Sig_n^k(-,z):(\Xz)^n\to\Xk$ where 
$$\Sig_n^k(a,z)\defi\Sig_n^k(b,z)\qfa\ b\in C(a)$$
is well defined. 
It follows from the triangle inequality and our assumption that $\mu$ satisfies~\eqref{assump:0} that 
\begin{equation}\label{eq:lypexp}
    \int_{(\Xz)^n}\norm{\Sig^k_n(a,z)}\dv{\mu^{\otimes n}}a\leq\norm z + \ka n,
\end{equation}
where $\ka$ is the first moment of $\mu$. 
\begin{lem}\label{lem:norm is contracted}
For all $k\in\bN$ there exist $0<c<1$, $M>0$ and $n'_0>0$ such that for all $z\in \Xk$ one has 
$$\int_\Xk  \norm{x}\dv{\om^k_{n'_0,z}}x<c\norm{z}+M.$$
\end{lem}
\begin{proof}
Let $n_0$ and $\al$ be as in Lemma~\ref{lem:postive proportion of primes}. 
Choose $\eps_0>0$ small enough so that the constant $C_{\eps_0,\al}\leq 2$.
Let $y\defi\al(1-\eps_0)/2n_0$ and $n\in\bN$. 
Begin by dividing the set $(\Xz)^n$ into two pieces as follows 
\begin{align*}
    S^k_{n,z}&\defi\set{a\in (\Xz)^n: \cY_n^k(a,z)\leq y n}\\
    R^k_{n,z}&\defi (\Xz)^n\setminus S^k_{n,z}.
\end{align*}
We now split the integral into two corresponding pieces using the definition of the measures $\om^k_{n,z}$
$$\int_{\Xk}\norm {x}\dv{\om^k_{n,z}}x=I_1+I_2,$$
where
$$I_1\defi\int_{S^k_{n,z}}\norm{\Sig_n^k(a,z)}\dv{\mu^{\otimes n}} a\qand
I_2\defi\int_{R^k_{N,z}}\norm{\Sig_n^k(a,z)}\dv{\mu^{\otimes n}} a.$$
By Lemma~\ref{lem:postive proportion of primes} and~\eqref{eq:lypexp} there exists $n_0>0$ such that for all $n>n_0$ and $z\in\Xk$ we have that 
$$I_1\leq \mu^{\ast n}(S^k_{n,z})(\norm z + \ka n)\le 2\exp\Big(-\frac{\al\eps_0^2}{2n_0}n\Big)(\norm z + \ka n)$$
and
$$I_2\leq\mu^{\ast n}(R^k_{n,z})\Bigl(\frac{1}{k^{\ceil{yn}}}\norm z +\ka n\Bigr)\leq\frac{1}{k^{\ceil{yn}}}\norm z +\ka n.$$
Hence, choosing $n'_0>n_0$ large enough so that 
$$2\exp\Big(-\frac{\al\eps_0^2}{2n_0}n'_0\Big)+\frac{1}{k^{\ceil{yn'_0}}}<1$$
we get the claim of the lemma. 
\end{proof}

\begin{cor}\label{cor:geometric series}
For all $k\in\bN$, there exist constants $0<c<1$, $M'>0$ and $n'_0>0$ such that for all $n\ge n'_0$ and $z\in\Xk$ one has
$$\int_\Xk\norm x\dv{\om^k_{n,z}} x\leq c^{\floor{n/n'_0}} \norm z +M'.$$
\end{cor}
\begin{proof}
Let $c$, $M$ and $n'_0$ be as in Lemma~\ref{lem:norm is contracted}. Let $\on T$ be the operator defined for measurable functions $f:\Xk\to\bR$ by 
$$\on T f (z)\defi\int_\Xk f\dv{\om^k_{n'_0,z}}{}.$$
Note that, if $N(z)\defi\norm z$, then the conclusion of Lemma~\ref{lem:norm is contracted} says that 
$$\on T N(z)\leq cN(z)+M.$$
It follows that 
$$\on T^iN(z)\leq c^i N(z)+M(c^{i-1}+\dots+c+1)\leq c^i N(z)+2M.$$
By noting that 
$$\on T^iN(z)=\int_\Xk\norm x\dv{\om^k_{in'_0,z}} x$$
we get that 
\begin{equation}\label{eq:multofi}\int_\Xk\norm x\dv{\om^k_{in'_0,z}} x\leq c^i \norm z +2M\qfa\ i\in\bN.\end{equation}
Suppose that $n\geq n'_0$ and let $j\in\set{0,\dots,n'_0-1}$ be such that $n=i_0n'_0+j$ for some $i_0\in\bN$. 
Note that $i_0=\floor{n/n_0}$.
Next we write 
$$\int_\Xk\norm x\dv{\om^k_{n,z}} x=\int_{(\Xz)^j}\int_\Xk \norm x \dv{\om^k_{in'_0,\Sig^k_j(a,z)}}x\dv{\mu^{\otimes j}}a.$$
Hence using~\eqref{eq:lypexp} and~\eqref{eq:multofi} we get that 
$$\int_\Xk\norm x\dv{\om^k_{n,z}} x\leq c^{\floor{n/n'_0}} (\norm z +\ka j)+2M.$$
Since $\ka$ and $j$ are bounded we may take $M'\defi 2M+\ka j$ and this is the conclusion of the lemma.
\end{proof}


We can now prove Proposition~\ref{prop:redux}. 
\begin{proof}[Proof of Proposition~\ref{prop:redux}]
Let $c$, $M'$ and $n'_0$ be as in Corollary~\ref{cor:geometric series} and $\eps>0$ be arbitrary.
Let $$K_\eps\defi\set{z\in\Xk : \norm z \le 2M'/\eps}.$$
It follows that $\bone_{\Xk\sm K_\eps}(z)\leq\frac{\eps}{2M'}\norm z$ for all $z\in \Xk$.
Hence, it follows from the conclusion of Corollary~\ref{cor:geometric series} that 
$$\om^k_{n,z}(\Xk\sm K_\eps)\leq\frac{\eps}{2M'}(c^{\floor{n/n'_0}}\norm z +M')=\frac{\eps c^{\floor{n/n_0}} }{2M'}\norm z+\frac{\eps}{2}$$
for all $n\ge n'_0$. The claim of the proposition follows as soon as $n$ is large enough so that $c^{\floor{n/n'_0}}\norm z/M'<1$.
\end{proof}


\subsection*{Acknowledgements}
The author is greatly indebted to Ross Pinsky for his time and effort in reading an earlier draft
and providing encouragement along with many detailed and helpful comments.
This reasearch was funded by the ISF grants
numbers 357/13 and 871/17.
\bibliographystyle{amsalpha}
\bibliography{ref}
\end{document}